\documentclass[letterpaper, 10 pt, conference]{ieeeconf}  

\IEEEoverridecommandlockouts                              

\usepackage{graphics} 

\usepackage{amsmath} 
\usepackage{amssymb}  
\usepackage{amsthm}
\usepackage{url}

\newcommand{\R}{\mathbb{R}}
\newcommand{\p}{\mathbf{p}}
\newcommand{\ovP}{\overline{P}}
\newcommand{\ovOm}{\overline{\Omega}}

\newcommand{\f}{\mathbf{f}}

\newcommand{\x}{\mathbf{x}}

\newcommand{\E}[1]{\mathbb{E}[#1]}
\newcommand{\om}{\boldsymbol{\omega}}

\newcommand{\GRelax}[4]{G^{#1}_{#2}(#3,#4)}

\newcommand{\sumi}{\sum\limits_{i=1}^n}
\renewcommand{\P}[1]{\mathbb{P}(#1)}

\newcommand{\bbP}{\mathbb{P}}
\newcommand{\bbE}{\mathbb{E}}
\newcommand{\bom}{\boldsymbol{\omega}}

\newcommand{\bphi}{\boldsymbol{\phi}}
\newcommand{\bpsi}{\boldsymbol{\psi}}
\newcommand{\bs}[1]{{\mathbf{#1}}}
\newcommand{\commentout}[1]{}

\newtheorem{theorem}{Theorem}
\newtheorem{lemma}[theorem]{Lemma}

\newtheorem{corollary}[theorem]{Corollary}
\newtheorem{assumption}[theorem]{Assumption}

\theoremstyle{definition}
\newtheorem{definition}[theorem]{Definition}

\theoremstyle{remark}

\title{\LARGE \bf
Convex Relaxations for Nonlinear Stochastic Optimal Control Problems
}

\author{Yuanxun Shao$^{1}$, Dillard Robertson$^{1}$ and Joseph K. Scott$^{1}$
\thanks{$^{1}$ Department of Chemical and Biomolecular Engineering, Clemson University, Clemson, SC.
	{\tt\small jks9@clemson.edu}}   
}

\begin{document}

\maketitle
\thispagestyle{empty}
\pagestyle{empty}

\begin{abstract}
This article presents a new method for computing guaranteed convex and concave relaxations of nonlinear stochastic optimal control problems with final-time expected-value cost functions. This method is motivated by similar methods for deterministic optimal control problems, which have been successfully applied within spatial branch-and-bound (B\&B) techniques to obtain guaranteed global optima. Relative to those methods, a key challenge here is that the expected-value cost function cannot be expressed analytically in closed form. Nonetheless, the presented relaxations provide rigorous lower and upper bounds on the optimal objective value with no sample-based approximation error. In principle, this enables the use of spatial B\&B global optimization techniques, but we leave the details of such an algorithm for future work. 
\end{abstract}

\section{Introduction}
\label{Sec:Intro}
This article concerns the guaranteed global solution of the stochastic optimal control problem stated informally as
\begin{alignat}{2}
\label{Eq: Intro Opt}
&\min\limits_{\p\in \ovP\subset\R^{n_p}} \quad    &&\E{g(\p,\om,\x(t_f,\p,\om))}, \\
&\quad \ \text{s.t.} \quad    &&\bbP[\bs{h}(\p,\om,\x(t_f,\p,\om))\leq\bs{0}]\geq 1-\alpha, \nonumber
\end{alignat}
where $\mathbb{E}$ and $\bbP$ denote the expected value and probability over continuous random variables $\om\in\ovOm\subset\mathbb{R}^{n_{\om}}$, respectively, and $\x(t_f,\p,\om)$ is the solution of the nonlinear ordinary differential equations (ODEs)
\begin{alignat}{2}
\label{Eq: ODEs}
&\dot{\x}(t,\p,\om)		    	    &&  = \f(t,\p,\om,\x(t,\p,\om)),\\
&\x(t_0,\p,\om)						&&  = \x_0(\p,\om) \nonumber.
\end{alignat}
The decision vector $\p$ may represent a parameterized open-loop control trajectory, parameters in an explicit feedback controller embedded in \eqref{Eq: ODEs}, etc. Such problems arise in stochastic model predictive control \cite{Mesbah2016}, renewable energy systems \cite{Hakizimana:2016}, trajectory planning \cite{Blackmore:2011}, chemical process control \cite{Behrooz2017}, and many other applications.

For optimal control problems with deterministic objectives and constraints, a number of algorithms have recently been developed that can provide guaranteed global solutions \cite{Scott:DAEForum,Houska:2014,Zhao:2011,Papamichail:RigBBB}. In brief, these methods are predicated on effective algorithms for enclosing the reachable set of the dynamics on subintervals of the decision space. These enclosures can take the form of fixed interval bounds or other fixed sets \cite{Zhao:2011}, but are more commonly described by bounds that depend affinely or convexly on the decisions $\bs{p}$ \cite{Singer:nonlinear,Scott:2013b,Scott:2GODERel,Sahlodin:2011,Sahlodin:DTRODERel,Papamichail:RigBBB}. With such an enclosure, it is possible to construct convex relaxations of the optimal control problem on arbitrary subintervals of the decision space, and using these, to compute bounds on the optimal objective value on such subintervals. Finally, these bounds can be used within a generic spatial branch-and-bound (B\&B) algorithm \cite{Horst:GOtext1} to obtain a rigorous global solution.

To the best of our knowledge, no such guaranteed global optimization algorithm is available for the stochastic problem \eqref{Eq: Intro Opt}. In this case, a critical new challenge is that the expected-value and probability appearing in the objective and constraints cannot be evaluated analytically in closed form, and must be evaluated instead by sampling. In the context of the optimization approach outlined above, this is problematic because it is no longer possible to obtain guaranteed bounds on the optimal objective value. In fact, using only sample-based approximations, it is not even possible to bound the objective and constraint values at a given feasible point with finitely many computations.

In practice, this problem is most commonly addressed by replacing the objective and constraints in \eqref{Eq: Intro Opt} by sample-average approximations (SAA), resulting in a deterministic optimal control problem that can be solved using existing methods. However, SAA has several critical limitations that can lead to inaccurate solutions or excessive computational cost. First, it only guarantees convergence to a global solution as the sample size tends to infinity \cite{Shapiro2008}. Moreover, the number of samples required to achieve a high-quality solution in practice is unknown and can be quite large \cite{Verweij2003,Linderoth2006}. More importantly, a sufficient sample size is not known in advance. Thus, it is often necessary to solve several SAA problems with independent samples to assess solution accuracy, and to repeat the entire process if a larger sample size is deemed necessary \cite{Linderoth2006}. This is clearly problematic for nonconvex optimal control problems, where solving a single instance to global optimality is already demanding.

In this article, we take a first step towards extending the rigorous global optimization methods outlined above to the stochastic problem \eqref{Eq: Intro Opt}. Specifically, our main contribution is a new method for computing guaranteed convex and concave relaxations of the final-time expected-value objective function. As with the deterministic methods above, we rely on an existing method for computing time-varying bounds on the solutions of \eqref{Eq: ODEs} \cite{Scott:2013b}. However, we modify the method here to obtain lower and upper bounds that are convex and concave, respectively, with respect to both $\p$ and $\bom$. Through an application of Jensen's inequality, we then obtain a time-varying, $\p$-dependent convex enclosure of the mapping $t\mapsto \E{g(\p,\om,\x(t,\p,\om))}$. This relaxation method is similar in spirit to the so-called \emph{probability bounds} for dynamic systems in \cite{Enszer:2011,Enszer:2015}, but these works do not consider expected-value bounds and have not been applied in the context of optimization. Our method is also related to our recent work in \cite{Shao:2017} (Pre-print available at \url{osf.io/qab6n}), which described convex and concave relaxations of expected-value functions that do not depend on the solution of a dynamic system.

In the absence of chance constraints, the relaxation technique presented here can provide both lower and upper bounds on the optimal objective value of \eqref{Eq: Intro Opt}, without resorting to sample-based approximations. In principle, this enables the application of spatial B\&B to solve \eqref{Eq: Intro Opt} to guaranteed global optimality with no approximation error. However, we leave the details of such a B\&B algorithm, as well as the treatment of chance constraints, for future work. 

The remainder of this article is organized as follows. First, Section \ref{Sec:Problem Statement} gives a formal problem statement. Our new relaxation theory is then developed in two steps in Sections \ref{Sec: Relaxing the Dynamics} and \ref{Sec: Relaxing the Expected Value}. In Section \ref{Sec: Rigorous Bounds on Stochastic Optimal Control Problems}, we apply the developed relaxations to obtain computable upper and lower bounds on the optimal objective value of \eqref{Eq: Intro Opt} in the absence of chance constraints. In Section \ref{Sec: Numerical Example}, we demonstrate the proposed relaxation technique on a simple case study. Finally, Section \ref{Sec: Conclusion} provides concluding remarks.

\section{Problem Statement}
\label{Sec:Problem Statement}
Let $ I=[t_0,t_f] \subset \R$ be a time horizon of interest, let $\ovP \subset \R^{n_p} $ be a compact $ n_p $-dimensional interval of decision variables $\p$, and let $\bom$ be a random vector with probability density function (PDF) $p:\mathbb{R}^{n_{\omega}}\rightarrow\mathbb{R}$. We assume that $p$ is zero outside of a compact interval $\overline{\Omega}\subset \R^{n_{\omega}}$. Let $\x_0:\R^{n_p}\times\R^{n_{\omega}} \rightarrow \R^{n_x}$ and $\f:\R\times\R^{n_p}\times\R^{n_{\omega}}\times\R^{n_{x}} \rightarrow\R^{n_x} $ be locally Lipschitz continuous functions defining the dynamics \eqref{Eq: ODEs}. We assume that \eqref{Eq: ODEs} has a unique solution $\bs{x}(\cdot,\p,\om)$ on all of $I$ for every $(\p,\om)\in \ovP\times\ovOm$. Finally, let $g:\R^{n_p}\times\R^{n_{\omega}}\times\R^{n_{x}} \rightarrow\R$ and define
\begin{alignat}{1}
\mathcal{G}(\bs{p})&\equiv\bbE[g(\p,\om,\x(t_f,\p,\om))], \\
&= \int_{\ovOm}g(\p,\om,\x(t_f,\p,\om))p(\om)d\om,
\end{alignat}
which is assumed to exist for every $\p\in \ovP$.

We are interested in computing convex and concave relaxations, defined as follows.
\begin{definition}
\label{Def:Relaxation of functions}
Let $S\subset \R^n$ be convex and $h:S\rightarrow \R$. Functions $h^{cv},h^{cc}:S\rightarrow \R$ are \emph{convex and concave relaxations of $h$ on $S$}, respectively, if $h^{cv}$ is convex on $S$, $h^{cc}$ is concave on $S$, and 
\begin{alignat*}{1}
h^{cv}(\bs{s})\leq h(\bs{s})\leq h^{cc}(\bs{s}), \quad \forall \bs{s}\in S.
\end{alignat*}
\end{definition}

The objective of this article is to develop a method for computing convex and concave relaxations of $\mathcal{G}$ on any given subinterval of $\overline{P}$. Specifically, we are interested in relaxations of $\mathcal{G}$ itself, rather than any finite approximation of $\mathcal{G}$ via sampling, quadrature, etc. At the same time, the relaxations themselves must be finitely computable to be of value in the context of spatial B\&B.

The following general notation is used in the remainder of the article. For any $\bs{s}^L,\bs{s}^U\in\mathbb{R}^{n}$ with $\bs{s}^L\leq\bs{s}^U$, let $S=[\bs{s}^L,\bs{s}^U]$ denote the compact $n$-dimensional interval $\{\bs{s}\in\mathbb{R}^{n}:\bs{s}^L\leq\bs{s}\leq\bs{s}^U\}$. Moreover, for $\overline{S}\subset \mathbb{R}^{n}$, let $\mathbb{I}\overline{S}$ denote the set of all compact interval subsets $S$ of $\overline{S}$. In particular, let $\mathbb{IR}^n$ denote the set of all compact interval subsets of $\mathbb{R}^n$.

\section{Relaxing the Dynamics on $P\times\Omega$}
\label{Sec: Relaxing the Dynamics}
The first step in our relaxation procedure is to compute convex and concave relaxations of the function $G:\mathbb{R}^{n_p}\times\mathbb{R}^{n_{\omega}}\rightarrow\mathbb{R}$ defined by
\begin{alignat}{1}
\label{Eq: G Def}
G(\p,\om)\equiv g(\p,\om,\x(t_f,\p,\om)).
\end{alignat}
Specifically, we will show in \S\ref{Sec: Relaxing the Expected Value} that the desired relaxations of the expected value $\mathcal{G}(\p)=\bbE[G(\p,\om)]$ can be readily computed from convex and concave relaxations of $G$ \emph{jointly} with respect to $\p$ and $\om$. Assuming that $g$ is known in closed form, and hence amenable to standard relaxation techniques \cite{McCormick:1976,Sahinidis:ConvRel}, the only complication in computing such joint relaxations of $G$ is the presence of the terminal time state vector $\x(t_f,\p,\om)$ in \eqref{Eq: G Def}, which we naturally assume is not known in closed form. To deal with this, we will construct joint \emph{state relaxations} defined as follows.

\begin{definition}
\label{Def:State Relaxations}
Choose any intervals $P\in\mathbb{I}\ovP$ and $\Omega\in\mathbb{I}\ovOm$. Two functions $ \x^{cv}, \x^{cc}: I\times P\times \Omega \rightarrow\R^{n_x} $ are called \emph{state relaxations} for \eqref{Eq: ODEs} on $P\times\Omega$ if $ \x^{cv}(t ,\cdot ,\cdot ) $ and $ \x^{cc}(t , \cdot ,\cdot ) $ are, respectively, convex and concave relaxations of $ \x(t ,\cdot ,\cdot ) $ on $P\times\Omega$, for every $t\in I$.
\end{definition}

Several methods have been developed for computing state relaxations in the deterministic case \cite{Singer:nonlinear,Scott:2013b,Scott:2GODERel,Sahlodin:2011,Sahlodin:DTRODERel,Papamichail:RigBBB}. Here, we extend the method in \cite{Scott:2013b} to produce joint relaxations on $P\times\Omega$. However, because Definition \ref{Def:State Relaxations} makes no mathematical distinction between the decisions $\p$ and the RVs $\bom$, the extension is direct (our overall relaxation procedure treats $\p$ and $\bom$ differently beginning in \S\ref{Sec: Relaxing the Expected Value}).

The method in \cite{Scott:2013b} computes state relaxations as the solutions of an auxiliary system of ODEs. Defining this system requires particular kinds of relaxations of the functions $\bs{x}_0$, $\bs{f}$ and $g$ to be available, which we now assume. Although the following assumptions may seem restrictive, the required relaxations can be automatically constructed for nearly any functions $\bs{x}_0$, $\bs{f}$, and $g$ through a generalization of McCormick's relaxation technique, as discussed in detail in \cite{Scott:2013b,Scott:GenMCRel}.

\begin{assumption}
\label{Assum: Relaxation Functions}
Assume that the following functions are available for any intervals $P\times\Omega\in\mathbb{I}\ovP\times\mathbb{I}\ovOm$:
\begin{enumerate}
\item $\bs{x}_{0,P\times\Omega}^{cv},\bs{x}_{0,P\times\Omega}^{cc}:P\times\Omega\rightarrow\R^{n_{x}}$ are continuous convex and concave relaxations of $\bs{x}_0$ on $P\times\Omega$.
\item $\bs{f}^{cv}_{P\times\Omega},\bs{f}^{cc}_{P\times\Omega}:I\times P\times\Omega\times\mathbb{R}^{n_x}\times\mathbb{R}^{n_x}\rightarrow\mathbb{R}^{n_x}$ are Lipschitz continuous and satisfy the following condition: For any continuous $\bphi,\bpsi:I\times P\times \Omega\rightarrow\mathbb{R}^{n_x}$ and any fixed $t\in I$, the functions
\begin{alignat}{1}
(\p,\om)&\mapsto\bs{f}^{cv}_{P\times\Omega}(t,\p,\om,\bphi(t,\p,\om),\bpsi(t,\p,\om)),\\
(\p,\om)&\mapsto\bs{f}^{cc}_{P\times\Omega}(t,\p,\om,\bphi(t,\p,\om),\bpsi(t,\p,\om)),
\end{alignat}
are respectively convex and concave relaxations of
\begin{alignat}{1}
(\p,\om)&\mapsto\bs{f}(t,\p,\om,\x(t,\p,\om))
\end{alignat}
on $P\times \Omega$, provided that $\bphi(t,\cdot,\cdot)$ and $\bpsi(t,\cdot,\cdot)$ are respectively convex and concave relaxations of $\bs{x}(t,\cdot,\cdot)$ on $P\times\Omega$.
\item $g^{cv}_{P\times\Omega},g^{cc}_{P\times\Omega}:P\times\Omega\times\mathbb{R}^{n_x}\times\mathbb{R}^{n_x}\rightarrow\mathbb{R}^{n_x}$ are Lipschitz continuous and satisfy the following condition: For any continuous $\bphi,\bpsi:P\times \Omega\rightarrow\mathbb{R}^{n_x}$, the functions
\begin{alignat}{1}
(\p,\om)&\mapsto g^{cv}_{P\times\Omega}(\p,\om,\bphi(\p,\om),\bpsi(\p,\om)),\\
(\p,\om)&\mapsto g^{cc}_{P\times\Omega}(\p,\om,\bphi(\p,\om),\bpsi(\p,\om)),
\end{alignat}
are respectively convex and concave relaxations of
\begin{alignat}{1}
(\p,\om)&\mapsto g(\p,\om,\x(t_f,\p,\om))
\end{alignat}
on $P\times \Omega$, provided that $\bphi$ and $\bpsi$ are respectively convex and concave relaxations of $\bs{x}(t_f,\cdot,\cdot)$ on $P\times\Omega$.
\end{enumerate}
\end{assumption}

Under Assumption \ref{Assum: Relaxation Functions}, the following theorem provides state relaxations for \eqref{Eq: ODEs} and the desired relaxations of $G$.


\begin{theorem}
\label{Thm:State and G Rels}
Choose any $P\times\Omega\in\mathbb{I}\ovP\times\mathbb{I}\ovOm$ and define the auxiliary system ODEs:
	\begin{alignat}{2}
	&\dot{\x}^{cv}(t,\p,\om)		    	    &&  = \f^{cv}_{P\times\Omega}(t,\p,\om, \x^{cv}(t,\p,\om), \x^{cc}(t,\p,\om)), \nonumber\\
	&\dot{\x}^{cc}(t,\p,\om)		    	    &&  = \f^{cc}_{P\times\Omega}(t,\p,\om, \x^{cv}(t,\p,\om), \x^{cc}(t,\p,\om)), \nonumber\\
	&\x^{cv}(t_0,\p,\om)						&&  = \x^{cv}_{0,P\times\Omega}(\p,\om), \nonumber\\
	\label{Eq:State Rel ODEs}
	&\x^{cc}(t_0,\p,\om)						&&  = \x^{cc}_{0,P\times\Omega}(\p,\om),
	\end{alignat}
	for all $ (t,\p, \om) \in I \times P\times\Omega$. This system has unique solutions $\bs{x}^{cv},\bs{x}^{cc}:I\times P\times\Omega\rightarrow\mathbb{R}^{n_x}$, and these solutions are state relaxations for \eqref{Eq: ODEs} on $P\times \Omega$. Moreover, the functions
	\begin{alignat*}{1}
	G^{cv}_{P\times\Omega}(\bs{p},\bom) &\equiv g_{P\times\Omega}^{cv}(\p,\bom,\x^{cv}(t_f,\p,\om),\x^{cc}(t_f,\p,\om)), \\
	G^{cc}_{P\times\Omega}(\bs{p},\bom) &\equiv g_{P\times\Omega}^{cc}(\p,\bom,\x^{cv}(t_f,\p,\om),\x^{cc}(t_f,\p,\om)),
	\end{alignat*}
	are convex and concave relaxations of $G$ on $P\times\Omega$.
\end{theorem}

\begin{proof}
Under Conditions 1 and 2 of Assumption \ref{Assum: Relaxation Functions}, a direct application of Theorem 4.1 in \cite{Scott:2013b} with $P:= P\times \Omega$ ensures that $\bs{x}^{cv}$ and $\bs{x}^{cc}$ exist, are unique, and are state relaxations for \eqref{Eq: ODEs} on $P\times \Omega$. Thus, Condition 3 of Assumption \ref{Assum: Relaxation Functions} can be applied with $\bphi=\x^{cv}(t_f,\cdot,\cdot)$ and $\bpsi=\x^{cc}(t_f,\cdot,\cdot)$, and it follows that $G_{P\times\Omega}^{cv}$ and $G_{P\times\Omega}^{cc}$ are convex and concave relaxations of $G$ on $P\times\Omega$.
\end{proof}

Once the relaxations in Assumption \ref{Assum: Relaxation Functions} have been constructed (see \cite{Scott:2013b}), the initial value problem \ref{Eq:State Rel ODEs} can be solved for any $(\p,\bom)\in P\times \Omega$ using any standard ODE solver, after which $G_{P\times\Omega}^{cv}$ and $G_{P\times\Omega}^{cc}$ can be directly evaluated.

\section{Relaxing the Expected Value on $P$}
\label{Sec: Relaxing the Expected Value}
In this section, we develop a method for computing convex and concave relaxations of the expected cost function
\begin{alignat}{1}
\label{Eq: mathcalG ReDef}
\mathcal{G}(\p)=\bbE[G(\p,\om)] = \bbE[g(\p,\om,\x(t_f,\p,\om))]
\end{alignat}
on any given $P\in\mathbb{I}\ovP$. In light of Theorem \ref{Thm:State and G Rels}, we assume throughout this section that relaxations $G_{P\times\Omega}^{cv}$ and $G_{P\times\Omega}^{cc}$ of $G$ are available on any desired subinterval $P\times\Omega\in\mathbb{I}\ovP\times\mathbb{I}\ovOm$.

To begin, note that for any $P\in \mathbb{I}\ovP$ and $\p\in P$, $\mathcal{G}(\p)$ is bounded from above and below by the values $\bbE[G^{cc}_{P\times\ovOm}(\p,\om)]$ and $\bbE[G^{cv}_{P\times\ovOm}(\p,\om)]$, respectively, as a trivial consequence of integral monotonicity. Moreover, these functions can readily be shown to be concave and convex on $P$, respectively. However, relaxations defined in this way are of no value for B\&B global optimization since they must be evaluated by sampling in general, and so guaranteed bounds cannot be computed from such relaxations finitely. To overcome this limitation, we follow the technique recently proposed for standard stochastic programs (rather than optimal control problems) in \cite{Shao:2017}. Namely, we apply Jensen's inequality to pass the expectation operator inside the relaxation functions $G^{cc}_{P\times\ovOm}$ and $G^{cv}_{P\times\ovOm}$.

\begin{lemma}[Jensen's inequality]
\label{Lem:Jensen's Inequality}
Let $\Omega\subset\overline{\Omega}$ be convex and let $h:\Omega\rightarrow \R$. If $h$ is convex and $\E{h(\bom)}$ exists, then $\E{h(\bom)}\geq h(\E{\bom})$. If $h$ is concave, then $\E{h(\bom)}\leq h(\E{\bom})$.
\end{lemma}
\begin{proof}
See Proposition 1.1 in \cite{Perlman:1974}.
\end{proof}

Although we could apply Jensen's inequality directly to the relaxations $G^{cv}_{P\times\ovOm}$ and $G^{cc}_{P\times\ovOm}$ on the whole uncertainty set $\ovOm$, this introduces conservatism in the resulting relaxations that cannot be controlled. In particular, relaxations defined in this way may not converge to $\mathcal{G}$ as the interval $P$ tends towards a singleton $[\p,\p]$ which is required for the convergence of spatial B\&B algorithms \cite{Horst:GOtext1}. Thus, we instead apply Jensen's inequality on a partition of $\ovOm$ that can be refined as needed.

\begin{definition}
\label{Def:Partition}
	A collection $\Phi=\{ \Omega_i \}_{i=1}^{n}$ of intervals $\Omega_i\in\mathbb{I}\ovOm$ is called an \emph{interval partition of $\ovOm$} if $\ovOm=\cup_{i=1}^n\Omega_i$ and $\mathrm{int}(\Omega_i)\cap\mathrm{int}(\Omega_j)=\emptyset$ for all distinct $i$ and $j$.
\end{definition}

\begin{definition}
\label{Def:Conditional EVal}
For any measurable $\Omega\subset\ovOm$, let $\mathbb{P}(\Omega)$ denote the probability of the event $\bom\in\Omega$, and let $\mathbb{E}[\cdot|\Omega]$ denote the conditional expected value conditioned on the event $\bom\in\Omega$.
\end{definition}

The following theorem provides the desired relaxations of $\mathcal{G}$. 

\begin{theorem}
\label{Thm: Relaxation of EVal}
	Let $\Phi=\{ \Omega_i \}_{i=1}^{n}$ be an interval partition of $\ovOm$. For every $P\in\mathbb{I}\ovP$ and every $\p\in P$, define 
	\begin{alignat}{1}
	\label{Eq:FcvXP definition}
	\mathcal{G}^{cv}_{P \times \Phi}(\p) &\equiv  \sumi \P{\Omega_i}\GRelax{cv}{P\times \Omega_i}{\p}{\mathbb{E}[\om| \Omega_i]}, \\
	\label{Eq:FccXP definition}
	\mathcal{G}^{cc}_{P \times \Phi}(\p) &\equiv  \sumi \P{\Omega_i}\GRelax{cc}{P\times \Omega_i}{\p}{\mathbb{E}[\om| \Omega_i]}.
	\end{alignat}
	Then $\mathcal{G}^{cv}_{P \times \Phi}$ and $\mathcal{G}^{cc}_{P \times \Phi}$ are convex and concave relaxations of $\mathcal{G}$ on $P$, respectively.\end{theorem}

\begin{proof}
By the law of total expectation (Proposition 5.1 in \cite{Ross:2002}), $\mathcal{G}(\p)$ can be expressed for any $\p\in \ovP$ as
\begin{alignat}{1}
\label{law of total expectation for F(x)}
\mathcal{G}(\p) & =\E{G(\p,\bom)}=\sumi \P{\Omega_i}\mathbb{E}[G(\p,\bom)|\Omega_i].
\end{alignat}
Thus, for any $P\in\mathbb{I}\ovP$ and any $\p\in P$, integral monotonicity and Jensen's inequality give,
\begin{alignat}{1}
\mathcal{G}(\p) & \geq \sumi \P{\Omega_i}\mathbb{E}[\GRelax{cv}{P\times \Omega_i}{\p}{\bom} | \Omega_i], \\
&\geq \sumi \P{\Omega_i}\GRelax{cv}{P\times \Omega_i}{\p}{\mathbb{E}[\bom | \Omega_i]}, \\
&=\mathcal{G}^{cv}_{P\times\Phi}(\p).
\end{alignat}
Noting that $\mathcal{G}^{cv}_{P\times\Phi}$ is a sum of convex functions on $P$, it must be convex itself, and is therefore a convex relaxation of $\mathcal{G}$ on $P$. The proof for $\mathcal{G}^{cc}_{P\times\Phi}$ is analogous.
\end{proof}

By considering an exhaustive partition of $\ovOm$, Theorem \ref{Thm: Relaxation of EVal} provides relaxations that are valid for the true expected value $\mathcal{G}$, rather than a finite approximation obtained via sampling or otherwise. Moreover, in contrast to sample-based approaches, the relaxations in Theorem \ref{Thm: Relaxation of EVal} can be evaluated finitely provided that the probabilities $\bbP[\Omega_i]$ and conditional expectations $\bbE[\bom|\Omega_i]$ are computable. This is clearly true if $\bom$ is uniformly distributed. On the other hand, directly evaluating these quantities for more general RVs often requires difficult multidimensional integrations. However, the article \cite{Shao:2017} presents an approach that avoids these computations for a variety of common distributions by using well-known change-of-variables formulas to reformulate the RVs of interest as uniform RVs (e.g., using the inverse CDF transform). Since our relaxation theory does not require linearity or convexity assumptions, using such transformations poses no additional difficulties.

Clearly, the use of an exhaustive partition of $\ovOm$ in Theorem \ref{Thm: Relaxation of EVal} is also a drawback, since in practice it limits our approach to models with a modest number of RVs. However, note that Theorem \ref{Thm: Relaxation of EVal} provides valid relaxations on \emph{any} partition $\Phi$, no matter how coarse. Thus, one can use partitions appropriate for any desired level of accuracy. In particular, in the context of spatial B\&B, coarse partitions may be sufficient to eliminate large regions of the search space from consideration, with fine partitions being required only in the vicinity of global optimizers. We leave the issue of effective partitioning rules for future work. We also leave for future work a formal analysis of the convergence of $\mathcal{G}^{cv}_{P\times\Phi}$ and $\mathcal{G}^{cc}_{P\times\Phi}$ to $\mathcal{G}$ as $P$ tends towards a singleton $[\p,\p]$. However, note that Theorem 5.4 in \cite{Shao:2017} shows that the expected-value relaxation strategy used in Theorem \ref{Thm: Relaxation of EVal} inherits the convergence properties of the integrand relaxations $G_{P\times\Omega_i}^{cv}$ and $G_{P\times\Omega_i}^{cc}$, provided that $\ovOm$ is partitioned sufficiently quickly as the width of $P$ diminishes. In turn, the convergence properties of $G_{P\times\Omega_i}^{cv}$ and $G_{P\times\Omega_i}^{cc}$ depend on those of the relaxations defined in Assumption \ref{Assum: Relaxation Functions}, and have been studied in detail in \cite{Schaber:2017,Schaber:Thesis}.

\section{Rigorous Bounds on Stochastic Optimal Control Problems}
\label{Sec: Rigorous Bounds on Stochastic Optimal Control Problems}
In this section, we apply the results of Sections \ref{Sec: Relaxing the Dynamics} and \ref{Sec: Relaxing the Expected Value} to establish computable upper and lower bounds on the optimal objective value of the stochastic optimal control problem
\begin{alignat}{2}
\label{Eq: Opt No Chance}
&\min\limits_{\p\in \ovP} \quad &&\mathcal{G}(\p)\equiv\E{g(\p,\om,\x(t_f,\p,\om))}.
\end{alignat} 
Specifically, in order to solve \eqref{Eq: Opt No Chance} to gauranteed global optimality using spatial B\&B, it is necessary to provide the B\&B routine with upper and lower bounds for \eqref{Eq: Opt No Chance} restricted to any given interval $P\in\mathbb{I}\ovP$. To obtain a lower bound, the standard approach is to minimize a convex relaxation over $P$, which is available via Theorem \ref{Thm: Relaxation of EVal}. To obtain an upper bound, one common approach is simply to  evaluate the objective at any feasible $\p\in P$. However, this is not possible for \eqref{Eq: Opt No Chance} because $\mathcal{G}(\p)$ cannot be evaluated finitely. Thus, a different approach is required to obtain a valid upper bound. In the following corollary, we employ a special case of the concave relaxation defined in Theorem \ref{Thm: Relaxation of EVal}.

\begin{corollary}
\label{Cor:Lower Bounding Problem}
Let $\Phi=\{ \Omega_i \}_{i=1}^{n}$ be an interval partition of $\ovOm$ and define the shorthand $\bom_i\equiv \bbE[\bom|\Omega_i]$. For any $P\in\mathbb{I}\ovP$, a lower bound on the optimal objective value of \eqref{Eq: Opt No Chance} restricted to $P$ is given by the optimal objective value of the following deterministic convex optimal control problem:
\begin{alignat}{2}
\label{Eq: Opt No Chance Relaxation}
&\min\limits_{\p\in P} \ && \sumi \P{\Omega_i}g^{cv}_{P\times \Omega_i}(\p,\bom_i,\x^{cv}_i(t_f,\p), \x^{cc}_i(t_f,\p)),
\end{alignat}
where $\x^{cv}_i(t,\p)$ and $\x^{cc}_i(t,\p)$ are the unique solutions of the $n$ independent systems of ODEs given for each $i\in\{1,\ldots,n\}$ by
\begin{alignat}{2}
	&\dot{\x}_i^{cv}(t,\p) && = \f^{cv}_{P\times\Omega_i}(t,\p,\om_i,\x^{cv}_i(t,\p), \x^{cc}_i(t,\p)), \nonumber\\
	&\dot{\x}_i^{cc}(t,\p) && = \f^{cc}_{P\times\Omega_i}(t,\p,\om_i,\x^{cv}_i(t,\p), \x^{cc}_i(t,\p)), \nonumber\\
	&\x^{cv}_i(t_0,\p) &&  = \x^{cv}_{0,P\times\Omega_i}(\p,\om_i), \nonumber\\
	\label{Eq:LBP ODEs}
	&\x^{cc}_i(t_0,\p) &&  = \x^{cc}_{0,P\times\Omega_i}(\p,\om_i).
	\end{alignat}
	Moreover, for any $\p\in P$, an upper bound on the optimal objective value of \eqref{Eq: Opt No Chance} restricted to $P$ is given by the value
	\begin{alignat}{1}
	\label{Eq: upper bound}
	\sumi \P{\Omega_i}g^{cc}_{[\p,\p]\times\Omega_i}(\p,\bom_i,\x^{cv}_i(t_f,\p), \x^{cc}_i(t_f,\p)),
\end{alignat}
where $\x^{cv}_i(t,\p)$ and $\x^{cc}_i(t,\p)$ now denote the unique solutions of the $n$ independent systems of ODEs given for each $i\in\{1,\ldots,n\}$ by
\begin{alignat}{2}
	&\dot{\x}_i^{cv}(t,\p) && = \f^{cv}_{[\p,\p]\times\Omega_i}(t,\p,\om_i,\x^{cv}_i(t,\p), \x^{cc}_i(t,\p)), \nonumber\\
	&\dot{\x}_i^{cc}(t,\p) && = \f^{cc}_{[\p,\p]\times\Omega_i}(t,\p,\om_i,\x^{cv}_i(t,\p), \x^{cc}_i(t,\p)), \nonumber\\
	&\x^{cv}_i(t_0,\p) &&  = \x^{cv}_{0,[\p,\p]\times\Omega_i}(\p,\om_i), \nonumber\\
	\label{Eq:UBP ODEs}
	&\x^{cc}_i(t_0,\p) &&  = \x^{cc}_{0,[\p,\p]\times\Omega_i}(\p,\om_i).
	\end{alignat}
\end{corollary}
\begin{proof}
By Theorem \ref{Thm: Relaxation of EVal}, a lower bound on the optimal objective value of \eqref{Eq: Opt No Chance} restricted to $P$ is given by the optimal objective value of the convex problem $\min_{\p\in P}\mathcal{G}^{cv}_{P\times\Phi}(\p)$. Applying the definitions of $\mathcal{G}^{cv}_{P\times\Phi}$ and $G^{cv}_{P\times\Omega_i}$ from Theorem \ref{Thm:State and G Rels} and Theorem \ref{Thm: Relaxation of EVal}, this lower bounding problem is equivalent to \eqref{Eq: Opt No Chance Relaxation}.

For the upper bound, note that $\p$ is feasible in \eqref{Eq: Opt No Chance}, so it suffices to bound $\mathcal{G}(\p)$. Applying Theorem \ref{Thm: Relaxation of EVal} with the degenerate interval $P=[\p,\p]$, it follows that $\mathcal{G}(\p)$ is bounded above by $\mathcal{G}^{cc}_{[\p,\p]\times\Phi}(\p)$. Again applying the definitions of $\mathcal{G}^{cc}_{[\p,\p]\times\Phi}$ and $G^{cc}_{[\p,\p]\times\Omega_i}$ from Theorems \ref{Thm:State and G Rels} and \ref{Thm: Relaxation of EVal}, this upper bound is equivalent to \eqref{Eq: upper bound}.
\end{proof}

\section{Numerical Example}
\label{Sec: Numerical Example}
The following nonlinear ODEs describe a negative resistance circuit consisting of an inductor, a capacitor, and a resistive element in parallel, where $ x_1 $ is the current through the inductor and $ x_2 $ is the voltage across the capacitor \cite{Khalil:NonlinearSystems}:
	\begin{alignat}{2}
	\label{Eq: Circuit ODEs}
	& \dot{x}_1 &&  =p_1 x_2,\\
	& \dot{x}_2	&&  =-p_2(x_1-x_2+x_2^3/3), \nonumber \\
	& x_{0,1}	&&  = \omega_1, \nonumber \\
	& x_{0,2}	&&  = \omega_2. \nonumber
	\end{alignat}	
We take the initial conditions to be independent random variables, both following a truncated normal distribution with mean $\mu_1=\mu_2=1$, standard deviation $\sigma_1=\sigma_2=0.1$, and truncation range $\ovOm=[\mu_1-3\sigma_1,\mu_1+3\sigma_1] \times [\mu_2-3\sigma_2,\mu_2+3\sigma_2]$. The parameters $ p_1 $ and $ p_2 $ are the inverses of the inductance and capacitance, respectively, and are scaled so that \eqref{Eq: Circuit ODEs} is dimensionless.

Consider the problem of relaxing the expected-value function $\mathcal{G}(\p) =\E{x_1(t_f,\p,\om)}$ on the set $ \p\in P =[0.1,0.3] \times [0.1,0.3]$ with $[t_0,t_f]=[0,5]$. Clearly, this function cannot be evaluated analytically, so standard relaxation techniques are not applicable. To apply Theorems \ref{Thm:State and G Rels} and \ref{Thm: Relaxation of EVal}, we considered interval partitions $\Phi=\{ \Omega_i \}_{i=1}^{n}$ of $ \ovOm $ consisting of 1, 16, and 64 uniform subintervals. We computed relaxations of the initial condition functions and right-hand side functions in \eqref{Eq: Circuit ODEs} satisfying Assumption \ref{Assum: Relaxation Functions} using generalized McCormick relaxations \cite{Scott:GenMCRel} as described in \cite{Scott:2013b}. For every interval $P\times \Omega_i$, the state relaxation system defined in Theorem \ref{Thm:State and G Rels} was solved at the point $(\p,\bom_i)=(\p,\bbE[\bom|\Omega_i])$ using the code \verb-CVODE- in the Sundials Matlab Toolbox \cite{Hindmarsh:SUNDAILS} with default tolerances. Finally, relaxations of $\mathcal{G}$ were constructed by summing the resulting relaxations of $x_1$ at $t_f$ weighted by the probabilities $\mathbb{P}(\Omega_i)$, as described in Theorem \ref{Thm: Relaxation of EVal}.

Figure \ref{Figure example 1} shows the resulting convex and concave relaxations, along with several final time solutions of \eqref{Eq: Circuit ODEs} computed for random samples of $\bom\in \ovOm$ and a sample-average approximation of $\mathcal{G}$ computed using 200 samples. Note that the relaxations enclose the true expected value as per Theorem \ref{Thm: Relaxation of EVal}, but need not enclose all solutions of \eqref{Eq: Circuit ODEs} for sampled $\bom\in \ovOm$. Figure \ref{Figure example 1} shows that the relaxations become tighter as the partition $\Phi$ is refined, but the improvement from 16 to 64 subintervals is minor. This suggests that the proposed method can provide reasonably tight relaxations using fairly coarse partitions of the uncertainty space.

	\begin{figure}[thpb]
		\centering

		\includegraphics{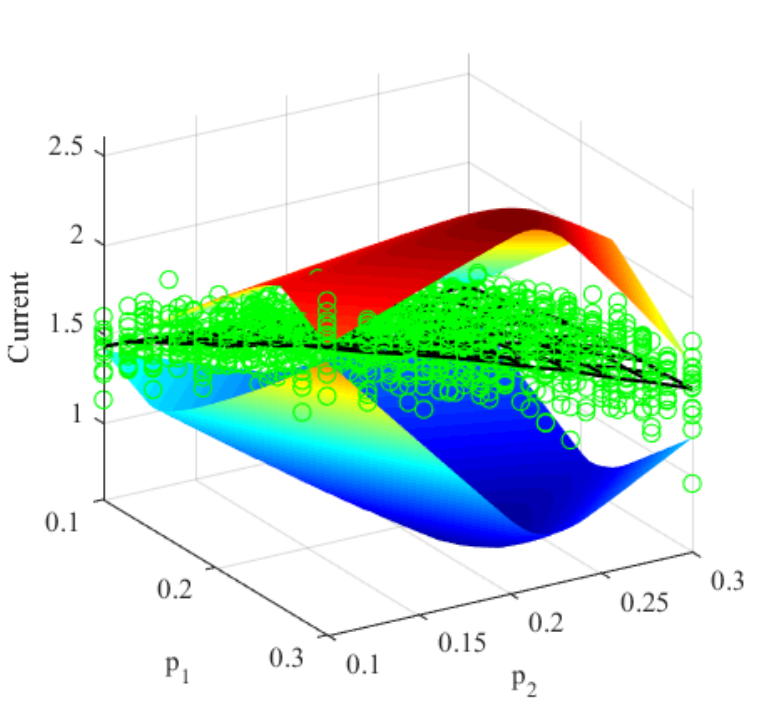}
		\includegraphics{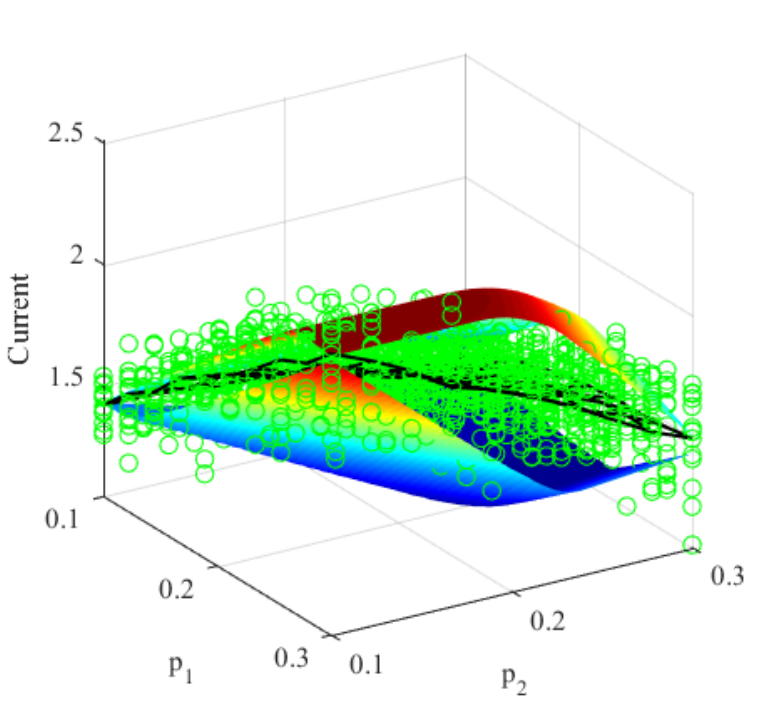}
		\includegraphics{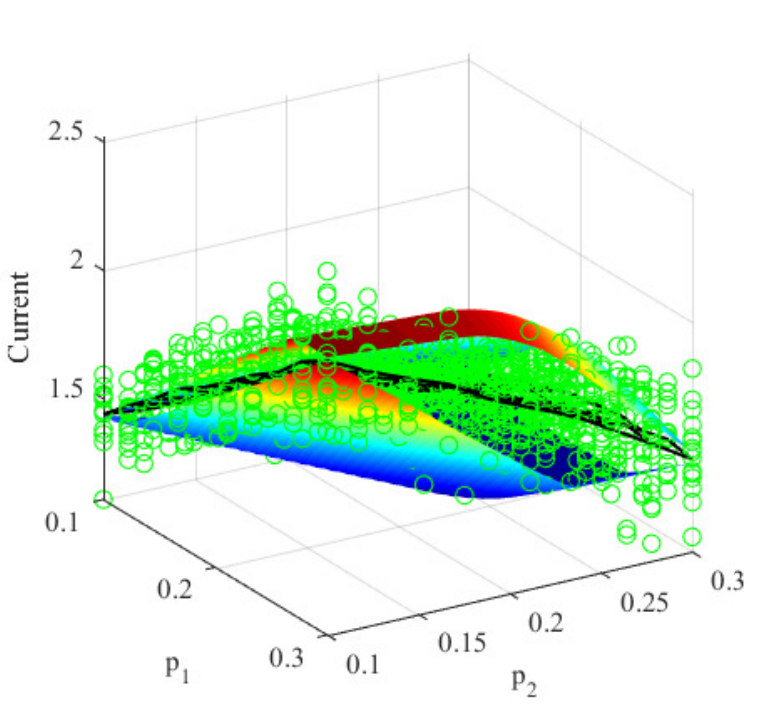}
		
		\caption{Convex and concave relaxations of $\mathcal{G}(\p)=\E{x_1(t_f,\p,\om)}$ on $P =[0.1,0.3] \times [0.1,0.3]$ (shaded surfaces) using partitions of $\overline{\Omega}$ into 1 (top), 16 (middle), and 64 (bottom) uniform subintervals, along with simulated values of $x_1(t_f,\p,\om)$ at sampled $\om$ values ($\circ$) and a sample-average approximation of $\mathcal{G}(\p)$ using 200 samples (black mesh).}
		\label{Figure example 1}
	\end{figure}

\section{Conclusions}
\label{Sec: Conclusion}
The main contribution of this article is a new approach for computing convex and concave relaxations of nonlinear stochastic optimal control problems with final-time expected-value cost functions. These relaxations can be used to compute rigorous upper and lower bounds on the optimal objective value of such problems restricted to any given subinterval of the decision space, as required for global optimization via spatial branch-and-bound. In this context, the key features of the presented relaxations are: (i) they provide valid bounds on the true objective function, without resorting to discrete approximations of either the expected value or the embedded dynamic model; (ii) they can be computed finitely, even when the objective function itself can only be approximated by sampling. Yet, both of these properties result from the use of an exhaustive partition of the uncertainty space (assumed compact), which limits the applicability of these relaxations to problems with a modest number of random variables. The presented case study suggests that the proposed relaxations can be made tight with a modest partition of the uncertainty space.



\footnotesize
\bibliographystyle{unsrt}
\bibliography{Master_YuanACC}

\end{document}